\documentclass[3p]{elsarticle}

\usepackage{amsmath,amsfonts,amsthm}
\usepackage{amssymb}
\usepackage[all]{xy}
\usepackage{enumerate}
\usepackage{mathrsfs}
\usepackage{graphicx}
\usepackage{bm}

\theoremstyle{plain}
\newtheorem{thm}{Theorem}[section]
\newtheorem{cor}[thm]{Corollary}
\newtheorem{lem}[thm]{Lemma}

\theoremstyle{definition}
\newtheorem{defn}[thm]{Definition}
\newtheorem{exmp}[thm]{Example}
\newtheorem{rem}[thm]{Remark}

\newcommand\dto{\dashrightarrow}

\newcommand\lto{\longrightarrow}
\newcommand\nto{\stackrel}

\def\NN{\mathbb{N}}
\def\ZZ{\mathbb{Z}}
\def\kk{k}
\def\PP{\mathbb{P}}
\def\RR{\mathbb{R}}

\def\AA{\mathbb{A}}
\def\NNbm{{\bm{N}}}
\def\MMbm{{\bm{M}}}

\newcommand\Sc{\mathscr{S}}
\newcommand\Hc{\mathscr{H}}
\newcommand\Zc{\mathcal{Z}}

\newcommand\Cc{\mathcal{C}}  
\newcommand\Tc{\mathscr{X}}
\newcommand\Pc{\mathscr{P}}

\newcommand\conv{\mathrm{conv}}

\def\ker{\mathrm{ker}}

\newcommand\Proj{\mathrm{Proj}}
\newcommand\proj{\mathrm{Proj}}
\def\deg{\mathrm{deg}}

\newcommand\Sym{\mathrm{Sym}}

\newcommand\ann{\mathrm{ann}}

\newcommand\length{\mathrm{length}}
\def\dim{\mathrm{dim}}

\newcommand\depth{\mathrm{depth}}

\newcommand\mult{\textnormal{mult}}
\newcommand\supp{\textnormal{supp}}
\newcommand\codim{\textnormal{codim}}
\newcommand\Spec{\textnormal{Spec}}

\newcommand\Mproj{\textnormal{Multiproj}}

\newcommand\ext{\textnormal{Ext}}

\newcommand\X{\textbf{X}}
\newcommand\T{\textbf{T}}

\def\f{\textbf{f}}

\def\x{\textbf{x}}

\newcommand\Nc{\mathcal{N}}
\newcommand\Supp{\mathrm{Supp}}

\newcommand\gen[1]{\left\langle#1\right\rangle}

\def\div{\textnormal{div}}
\newcommand\mat{\textnormal{Mat}}

\def\l.{\mathcal{L}_{\bullet}}
\def\ff.{\mathcal{F}_{\bullet}}
\def\a.{\mathcal{A}_\bullet}
\def\b.{\mathcal{B}_\bullet}

\def\k.{\mathcal{K}_{\bullet}}
\def\M.{\mathcal{M}_\bullet}
\def\Z.{\mathcal{Z}_\bullet}

\def\B.{\mathcal{B}_\bullet}

\def\OO{\mathcal O}

\def\LL{\mathcal L}

\def\LL{\mathcal L}

\def\HH{\mathcal H}

\def\pp{\mathfrak{p}}
\def\mm{\mathfrak{m}}
\def\Aa{\mathfrak{A}}

\def\kkk{\kappa}

\def\pp{\mathfrak{p}}
\def\qq{\mathfrak{q}}
\def\mm{\mathfrak{m}}
\def\aaa{\mathfrak{a}}
\def\bb{\mathfrak{b}}

\def\k.{\mathcal{K}_{\bullet}}

\def\kt.{\k.(\textbf{T};A[\textbf{T}])}
\def\ki.{\k.(\textbf{f};A[\textbf{T}])}

\def\P1{\PP^1}

\def\SSup{\mathfrak S}
\def\Region{{\mathfrak R _B}}
\def\C.{C_\bullet}
\def\D{\textnormal{\bf C}\textnormal{l}(\Xc)}
\def\DT{\textnormal{\bf C}\textnormal{l}(\Tc)}

\def\ee{\mathbf e}

\def\G{\textnormal{\bf G}}

\newcommand{\bfgamma}{\bm{\gamma}}
\newcommand{\bfrho}{\bm{\rho}}
\newcommand\cd{\textnormal{cd}}
\newcommand\grade{\textnormal{grade}}
\newcommand\SIR{\textnormal{Sym}_R (I)}
\newcommand\RIR{\textnormal{Rees}_R (I)}
\def\Xc{\mathscr X}
\def\Bc{\mathcal B}
\def\Mc{\mathcal M}
\newcommand\rad{\textnormal{rad}}

\def\F.{F_\bullet}

\newcommand\sg{\mathrm{sg}}

\newcommand\spann[1]{\left(#1\right)}
\newcommand\colimit[1]{\underset{#1}{\underset{\lto}{\text{lim}}} \ }
\newcommand\ot{\leftarrow}
\newcommand\paren[1]{\left(#1\right)}

\begin{document}

\title{The implicit equation of a multigraded hypersurface}

\author{Nicol\'{a}s Botbol}

\address{Departamento de Matem\'atica\\
FCEN, Universidad de Buenos Aires, Argentina \\
\& Institut de Math\'ematiques de Jussieu \\
Universit\'e de P. et M. Curie, Paris VI, France \\
E-mail address: nbotbol@dm.uba.ar
}

\thanks{This work was partially supported by the project ECOS-Sud A06E04, UBACYT X064, CONICET PIP 5617 and ANPCyT PICT
17-20569, Argentina.}


\begin{abstract}
In this article we analyze the implicitization problem of the image of a rational map $\phi: \Tc \dto \PP^n$, with $\Tc$ a toric variety of dimension $n-1$ defined by its Cox ring $R$. Let $I:=(f_0,\hdots,f_n)$ be $n+1$ homogeneous elements of $R$. We blow-up the base locus of $\phi$, $V(I)$, and we approximate the Rees algebra $\RIR$ of this blow-up by the symmetric algebra $\SIR$. We provide under suitable assumptions, resolutions $\Z.$ for $\SIR$ graded by the divisor group of $\Tc$, $\DT$, such that the determinant of a graded strand, $\det((\Z.)_\mu)$, gives a multiple of the implicit equation, for suitable $\mu\in \DT$. Indeed, we compute a region in $\DT$ which depends on the regularity of $\SIR$ where to choose $\mu$. We also give a geometrical interpretation of the possible other factors appearing in $\det((\Z.)_\mu)$. A very detailed description is given when $\Tc$ is a multiprojective space.
\end{abstract}

\maketitle
\medskip

\section{Introduction}


The interest in computing explicit formulas for resultants and discriminants goes back to B\'ezout, Cayley, Sylvester and many others in the eighteenth and nineteenth centuries. 
The last few decades have yielded a rise of interest in the implicitization of geometric objects motivated by applications in computer aided geometric design and geometric modeling as can be seen in for example in \cite{Kal90,MC92,MC92alg,AGR95,SC95}. This phenomena has been emphasized in the latest years due to the increase of computing power (cf.\ \cite{AS01,co01,BCD03,BuJo03,BC05,BCJ06,BD07,Bot08,BDD08,Bot09}). 

Under suitable hypotheses, resultants give the answer to many problems in elimination theory, including the implicitization of rational maps. 
In turn, both resultants and discriminants can be seen as the implicit equation of the image of a suitable map (cf.\ \cite{DFS07}). Lately, rational maps appeared in computer-engineering contexts, mostly applied to shape modeling using computer-aided design methods for curves and surfaces. A very common approach is to write the implicit equation as the determinant of a matrix whose entries are easy to compute.

Rational algebraic curves and surfaces can be described in several different ways, the most common being parametric and implicit representations. Parametric representations describe such curve or surface as the image of a rational map, whereas implicit representations describe it as the zero locus of a certain algebraic equation, e.g.\ as the zeros of a polynomial. Both representations have a wide range of applications in Computer Aided Geometric Design (CAGD), and depending on the problem one needs to solve, one or the other might be better suited. It is thus interesting to be able to pass from parametric representations to implicit equations. This is a classical problem and there are numerous approaches to its solution. For a good historical overview on this subject we refer the reader to \cite{SC95} and \cite{co01}.

Assume $x(s,t,u),\, y(s,t,u),\,z(s,t,u)$ and $w(s,t,u)$ are homogeneous polynomials of the same degree $d$ such that the parametrization
\begin{equation} \label{eqCassicStyle}
(s:t:u)\mapsto \paren{\frac{x(s,t,u)}{w(s,t,u)}:\frac{y(s,t,u)}{w(s,t,u)}:\frac{z(s,t,u)}{w(s,t,u)}:1}
\end{equation}
defines a surface in $\PP^3$.  The implicitization problem consists in the computation of a homogeneous polynomial $H(X,Y,Z,W)$ whose zero locus defines the scheme-theoretic closure of the surface given as the image of the parametrization.

However, it turns out that the implicitization problem is computationally difficult. The implicit equation can always be found using Gr\"obner bases. However, complexity issues mean that in practice, this method is rarely used in geometric modeling, especially in situations where real-time modeling is involved.  A more common method for finding the implicit equation is to eliminate $s,t,u$ by computing the resultant of the three polynomials
\[
x(s,t,u)-Xw(s,t,u),\ y(s,t,u)-Yw(s,t,u),\ z(s,t,u)-Zw(s,t,u).
\]
But in many applications, the resultant vanishes identically due to the presence of base points, which are common zeros in $\PP^2$ of all polynomials $x,y,z$ and $w$.

In consequence, the search of formulas for implicitization rational surfaces with base points is a very active area of research due to the fact that, in practical industrial design, base points show up quite frequently.  In \cite{MC92alg}, a perturbation is applied to resultants in order to obtain a nonzero multiple of the implicit equation. Many other approaches have been done in this direction. Lately, in \cite{BuJo03, BC05,BCJ06,BD07, BDD08,Bot09} it is shown how to compute the implicit equation as the determinant of the approximation complexes.

In \cite{KD06,BD07,BDD08,Bot09} different compactifications of the domain have been considered in order to erase base points, emphasizing the choice of the toric compactifications that better suits the monomial structure of the defining polynomials.

In this article, we present a method for computing the implicit equation of a hypersurface given as the image of a finite rational map $\phi: \Xc \dashrightarrow \PP^n$,  where $\Xc$ is a normal toric variety of dimension $n-1$. In \cite{BDD08} and \cite{Bot09}, the approach consisted in embedding $\Xc$ into a projective space, via a toric embedding. The need of the embedding comes from the necessity of a $\ZZ$-grading in the coordinate ring of $\Xc$, in order to study its regularity. 

Our contribution is to exploit the natural structure of the homogeneous coordinate ring of the toric variety where the map is defined. Thus, we present a novel approach to the method in \cite{BD07,BDD08} and \cite{Bot09} avoiding embedding the toric variety, and dealing with the homogeneous structure in a more natural way. 

Indeed, we deal with the multihomogeneous structure of the coordinate ring $R$ of $\Xc$, whose grading is given by the divisor class group of $\Xc$, $\DT$. The main motivations for our change of perspective are that it is more natural to deal with the original grading on $\Xc$, and that the embedding leads to an artificial homogenization process that makes the effective computation slower, as the number of variables to eliminate increases.

In Definition \ref{defRegion} we introduce the ``good'' region in $\DT$ where the approximation complex $\Z.$ is acyclic and the symmetric algebra $\SIR$ has no $B$-torsion. Indeed, we define for $\gamma\in \DT$, 
\[
 \Region(\gamma):=\bigcup_{0<k< \min\{m,\cd_B(R)\}} (\SSup_B(\gamma)-k\cdot \gamma)\subset \DT
\]
This goes in the direction of proving the main theorem of this article, Theorem \ref{CorToricImplicit}. Precisely it asserts that, if $\Xc$ is a ($n-1$)-dimensional non-degenerate toric variety over a field $\kk$, and $R$ its Cox ring (cf.\ Section \ref{CoxRing}), given a finite rational map $\phi: \Xc \dto \PP^{n}$ defined by $n+1$ homogeneous elements of degree $\gamma\in\D$ with $\dim(V(I))\leq 0$ in $\Xc$ and $V(I)$ is almost a local complete intersection off $V(B)$, we prove in Theorem  \ref{CorToricImplicit} that, 
\[
 \det((\Z.)_\mu)=H^{\deg(\phi)}\cdot G \in \kk[\T],
\]
for all $\mu\notin \Region(\gamma)$, where $H$ stands for the irreducible implicit equation of the image of $\phi$, and $G$ is relatively prime polynomial in $\kk[\T]$. This result generalizes to the setting of abstract toric varieties  theorems \cite[Thm. 5.7]{BuJo03} and \cite[Thm. 10 and Cor. 12]{Bot09}.


\medskip

\section{The Cox ring of a toric variety}\label{CoxRing}

Our main motivation for considering regularity in general $\G$-gradings comes from toric geometry. Among $\G$-graded rings, homogeneous coordinate rings of a toric varieties are of particular interest in geometry. In this section, we will overview some basic facts about Cox rings (cf.\ \cite{Cox95}). When $\Tc$ is a toric variety, $\G:=\DT$ is the divisor class group of $\Tc$, also called the Chow group of codimension one cycles of $\Tc$. The grading can be related geometrically with the action of this group on the toric variety, and hence, the graded structure on the ring can be interpreted in terms of global sections of the structural sheaf of $\Tc$ and in terms of character and valuations (cf.\ \cite{CoxTV}).

Henceforward, let $\Delta$ be a non-degenerate fan in the lattice $\NNbm\cong \ZZ^{n-1}$, and let $\Tc$ be a toric variety associated to $\Delta$. Write $\Delta(i)$ for the set of $i$-dimensional cones in $\Delta$. There is a bijection between the set $\Delta(i)$ and the set of closed $i$-dimensional subvarieties of $\Tc$. In particular, each $\rho\in \Delta(1)$ corresponds to the Weil divisor $D_\rho\in \ZZ^{\Delta(1)}\cong \ZZ^{n-1}$.

Suppose that $\rho_1,\hdots,\rho_{s}\in \Delta(1)$ are the one-dimensional cones of $\Delta$ and assume $\Delta(1)$ spans $\RR^{n-1}$. As before, $\eta_{\rho_i}$ denotes the primitive generator of $\rho_i \cap \NNbm$. There is a map $\MMbm\nto{\bfrho}{\to} \ZZ^{\Delta(1)}: m\mapsto \sum_{i=1}^s\gen{m,\eta_{\rho_i}}D_{\rho_i}$. We will identify $[D_{\rho_i}]$ with a variable $x_i$.

The divisor classes correspond to the elements of the cokernel $\DT$ of this map $\bfrho$, getting an exact sequence
\[
 0\to \ZZ^{n-1}\cong \MMbm \nto{\bfrho}{\lto} \ZZ^{s} \nto{\pi}{\lto} \DT \to 0.
\]

Set $R:=k[x_1\hdots,x_{s}]$. From the sequence above we introduce in $R$ a $\DT$-grading, which is coarser than the standard $\ZZ^{n-1}$-grading.

\medskip

To any non-degenerate toric variety $\Tc$, is associated an homogeneous coordinate ring, called the \textsl{Cox ring of $\Tc$} (cf.\ \cite{Cox95}). D.\ A.\ Cox defines (loc.\ cit.) the homogeneous coordinate ring of $\Tc$ to be the polynomial ring $R$ together with the given $\DT$-grading. We next discuss briefly this grading. A monomial $\prod x_i^{a_i}$ determines a divisor $D=\sum_i a_iD_{\rho_i}$, this monomial will be denoted by $\x^D$. For a monomial $\x^D\in R$ we define its degree as $\deg(\x^D)=[D]$ in $\DT$.

Cox remarks loc.\ cit.\ that the set $\Delta(1)$ is enough for defining the graded structure of $R$, but the ring $R$ and its graded structure does not suffice for reconstructing the fan. In order not to lose the fan information, we need to consider the irrelevant ideal
\[
 B:= \spann{\prod_{\eta_{\rho_i}\notin \sigma} x_i : \sigma \in \Delta},
\]
where the product is taken over all the $\eta_{\rho_i}$ such that the ray $\rho_i$ is not contained as an edge in any cone $\sigma\in \Delta$. Finally, the Cox ring of $\Tc$ will be the $\DT$-graded polynomial ring $R$, with the irrelevant ideal $B$.

\medskip

Given a $\DT$-graded $R$-module $M$, Cox constructs a quasi-coherent sheaf $M^\sim$ on $\Tc$ by localizing just as in the case of projective space, and he shows that finitely generated modules give rise to coherent sheaves. It was shown by Cox (cf.\ \cite{Cox95}) for simplicial toric varieties, and by Mustata in general (cf.\ \cite{Mus02}), that every coherent $\OO_{\Tc}$-module may be written as $M^\sim$, for a finitely generated $\DT$-graded $R$-module $M$.

For any $\DT$-graded $R$-module $M$ and any $\delta\in \DT$ we may define $M[\delta]$ to be the graded module with components $M[\delta]_\epsilon = M_{\delta+\epsilon}$ and we set
\[
 H^i_{\ast}(\Tc,M^\sim):=\bigoplus_{\delta\in \DT}H^i(\Tc,M[\delta]^\sim).
\]
We have $H^0(\Tc,\OO_\Tc(\delta)) = R_\delta$, the homogeneous piece of $R$ of degree $\delta$, for each $\delta \in \DT$. In fact each $H^i_\ast(\Tc,\OO_\Tc)$ is a $\ZZ^{n-1}$-graded $R$-module. For $i>0$, by (cf.\ \cite[Prop. 1.3]{Mus02}),
\begin{equation}\label{equLC-SCforS}
 H^i_{\ast}(\Tc,M^\sim)\cong H^{i+1}_B(M):= \colimit{j} \ext^i_R(R/B^j,R).
\end{equation}
and there is an exact sequence $0\lto H^0_B(M)\lto M \lto H^0_{\ast}(\Tc,M^\sim)\lto H^{1}_B(M)\lto 0$.

We will use these results in the following sections for computing the vanishing regions of local cohomology of Koszul cycles.


\medskip

\section{Regularity for $\G$-graded Koszul cycles and homologies}\label{GRegularity}
Throughout this article let $\G$ be a finitely generated abelian group, and let $R$ be a commutative noetherian $\G$-graded ring with unity. Let $B$ be an homogeneous ideal of $R$. Take $m$ a positive integer and let $\f:=(f_0,\hdots,f_m)$ be a tuple of homogeneous elements of $R$, with $\deg(f_i)=\gamma_i$, and set $\bfgamma:=(\gamma_0,\hdots,\gamma_m)$. Write $I=(f_0,\hdots,f_m)$ for the homogeneous $R$-ideal generated by the $f_i$.

As we mentioned, our main motivation for considering regularity in general $\G$-gradings comes from toric geometry. When $\Xc$ is a toric variety, $\G:=\DT$ is the divisor class group of $\Xc$. In this case, the grading can be related geometrically with the action of this group on the toric variety. Thus, it is of particular interest the case where $R$ is a polynomial ring in $n$ variables and $\G=\ZZ^n/K$, is a quotient of $\ZZ^n$ by some subgroup $K$ and quotient map $\pi$. Note that, if $M$ is a $\ZZ^n$-graded module over a $\ZZ^n$-graded ring, and $\G=\ZZ^n/K$, we can give to $M$ a $\G$-grading coarser than its $\ZZ^n$-grading. For this, define the $\G$-grading on $M$ by setting, for each $\mu\in \G$, $M_\mu:=\bigoplus_{d\in \pi^{-1}(\mu)}M_d$.

\medskip
Next, we present several results concerning vanishing of graded parts of certain modules. In our applications we will mainly focus on vanishing of Koszul cycles and homologies. We recall here what the support of a graded modules $M$ is.

In order to fix the notation, we state the following definitions concerning local cohomology of graded modules, and support of a graded modules $M$ on $\G$. Recall that the cohomological dimension $\cd_B(M)$ of a module $M$ is defined as $\cd_B(M):=\inf\{i\in \ZZ:\ H^j_B(M)= 0\ \forall j>i\}$ and that $\depth_B(M):=\inf\{i\in \ZZ:\ H^i_B(M)\neq 0\}$. We will write $\grade(B):=\depth_B(R)$.

\begin{defn}\label{defSuppGP}
Let $M$ be a $\G$-graded $R$-module, the support of the module $M$ in $\G$ is $\Supp_\G(M):=\{\mu \in \G:\ M_\mu \neq 0\}$.
\end{defn}

\begin{thm}\label{ThmRegHGral}
 For a complex $\C.$ of graded $R$-modules, assume that one of the following conditions holds
\begin{enumerate}
 \item For some $q\in \ZZ$, $H_j(\C.)=0$ for all $j<q$ and, $\cd_B(H_j(\C.))\leq 1$ for all $j>q$.  
 \item $\cd_B(H_j(\C.))\leq 1$ for all $j\in \ZZ$.
\end{enumerate}
Then for any $i$, 
\[
 \Supp_\G(H^i_B(H_j(\C.)))\subset \bigcup_{k\in\ZZ}\Supp_\G(H^{i+k}_B(C_{j+k})).
\]
\end{thm}
\begin{proof}
Consider the two spectral sequences that arise from the double complex $\check \Cc^\bullet_B \C.$. The first spectral sequence has as second screen $ _2'E^i_j = H^i_B (H_j(\C.))$. From $(1)$ or $(2)$, we deduce that $_2'E^i_j = 0$ for $j\neq 0$ and $i\neq 0,1$, and for either $j<0$ or for $j=0$ and $i\neq 0,1$. It follows that $_\infty 'E^i_j =\, _2'E^i_j = H^i_B (H_j(\C.))$.

The second spectral sequence has as first screen $ _1''E^i_j = H^i_B (C_j)$. 

By comparing both spectral sequences, we deduce that, for fixed $\mu \in \G$, the vanishing of $H^{i+k}_B (C_{j+k})_\mu$ for all $k$ implies the vanishing of $H^i_B (H_{j}(\C.))_\mu$.
\end{proof}

We see from Theorem \ref{ThmRegHGral} that much of the information on the supports of the local cohomologies of the homologies of a complex $\C.$ is obtained from the supports of the local cohomologies of the complex. For instance, if $\C.$ is a free resolution of a graded $R$-module $M$, the supports of the local cohomologies of $M$ can be controlled in terms of the supports of the local cohomologies of the base ring $R$, and the shifts appearing in the $C_i$'{}s.

In order to lighten the reading of this article, we extend the definition as follows 

\begin{defn}\label{defdefSuppSigma}
 Let $M$ be a graded $R$-module. For every $\gamma \in \G$, we define
\begin{equation}\label{defSuppSigma}
 \SSup_B(\gamma;M):= \bigcup_{k\geq 0}(\Supp_\G(H^{k}_B(M))+k\cdot\gamma).
\end{equation}
We will write $\SSup_B(\gamma):=\SSup_B(\gamma;R)$.
\end{defn}

For an $R$-module $M$, we denote by $M[\mu]$ the shifted module by $\mu\in \G$, with $M[\mu]_{\gamma}:=M_{\mu+\gamma}$, hence, $\SSup_B(\gamma;M[\mu])=\SSup_B(\gamma;M)-\mu$.

\medskip

We apply Theorem \ref{ThmRegHGral} in the particular case where $\C.$ is a Koszul complex and we bound the support of the local cohomologies of its homologies in terms of the sets $\SSup_B(\gamma;M)$. Indeed, let $M$ be a $\G$-graded $R$-module. Denote by $\k.^M$ the Koszul complex $\k.(\f;R)\otimes_R M$. If  $\f=\{f_0,\hdots,f_m\}$ and $f_i$ are $\G$-homogeneous elements of $R$ of the same degree $\gamma$ for all $i$, the Koszul complex $\k.^M$ is $\G$-graded with $K_i:=\bigoplus_{l_0<\cdots<l_i}R(-i\cdot \gamma)$. Let $Z_i^M$ and $B_{i}^M$ be the Koszul $i$-th cycles and boundaries modules, with the grading that makes the inclusions $Z_i^M, B_{i}^M\subset K_i^M$ a map of degree $0\in \G$, and set $H_i^M=Z_i^M/ B_{i}^M$.

\begin{cor}\label{Ex1Koszul}
If $\cd_B(H_i^M)\leq 1$ for all $i>0$. Then, for all $j\geq 0$
\[
\Supp_\G(H^i_B(H_j^M))\subset \bigcup_{k\geq 0}(\Supp_\G(H^{k}_B(M))+k\cdot\gamma) + (j-i)\cdot\gamma.
\]
\end{cor}
\begin{proof}
This follows by a change of variables in the index $k$ in Theorem \ref{ThmRegHGral}. Since $\C.$ is $\k.^M$ and $K_i^M:=\bigoplus_{l_0<\cdots<l_i}M(-i\cdot \gamma)$, we get that 
\[
\Supp_\G(H^i_B(H_j^M))\subset \bigcup_{k\geq 0}\Supp_\G(H^{k}_B(K^M_{k+j-i})) = \bigcup_{k\geq 0}(\Supp_\G(H^{k}_B(M)[(i-k-j)\cdot\gamma]).
\]
 The conclusion follows by shifting degrees.
\end{proof}

\begin{rem}\label{RemEx1Koszul}
 In the special case where $M=R$, we deduce that if $\cd_B(H_i)\leq 1$ for all $i>0$, then
\[
 \Supp_\G(H^i_B(H_j))\subset  \bigcup_{k\geq 0}(\Supp_\G(H^{k}_B(R))+k\cdot\gamma) + (j-i)\cdot\gamma, \quad \mbox{for all }i,j.
\]
\end{rem}

Thus, if $\cd_B(H_i^M)\leq 1$ for all $i>0$, then, for all $j\geq 0$
\[
\Supp_\G(H^i_B(H_j^M))\subset \SSup_B(\gamma;M) + (j-i)\cdot\gamma.
\]

\begin{rem}\label{Ex2Koszul}
Write $I:=(f_1,\hdots,f_m)$. The case $j=0$ in Remark \ref{RemEx1Koszul} gives 
\[
 \Supp_\G(H^i_B(R/I))\subset  \bigcup_{k\geq 0}(\Supp_\G(H^{k}_B(R))+(k-i)\cdot\gamma), \quad \mbox{for all }i.
\]
\end{rem}

The next result determines the supports of Koszul cycles in terms of the sets $\SSup_B(\gamma)$. This will be essential in our applications, since the acyclicity region of the $\Zc$-complex and the torsion of the symmetric algebra $\SIR$ will be computed directly from these regions.

\begin{lem}\label{LemMultiLCRGral} Assume $f_0,\hdots,f_m\in R$ are homogeneous elements of same degree $\gamma$. Write $I=(f_0,\hdots,f_m)$. Fix a positive integer $c$. If $\cd_B(R/I)\leq c$, then the following hold
\begin{enumerate}
 \item $\Supp_\G(H^i_B(Z_q))\subset  (\SSup_B(\gamma)+(q+1-i)\cdot \gamma)\cup(\bigcup_{k\geq 0} \Supp_\G(H^{i+k}_B(H_{k+q}))\cdot \gamma)$, for $i\leq c$ and all $q\geq 0$.
 \item $\Supp_\G(H^i_B(Z_q))\subset \SSup_B(\gamma)+(q+1-i)\cdot \gamma$, for $i>c$ and all $q\geq 0$.
\end{enumerate}
\end{lem}
\begin{proof} 
Consider $\k.^{\geq q}: 0\to K_{m+1}\to K_{m}\to \cdots\to K_{q+1}\to Z_{q}\to 0$ the truncated Koszul complex. The double complex $\check \Cc^\bullet_B(\k.^{\geq q})$ gives rise to two spectral sequences. The first one has second screen $_2'E^i_j = H^i_B (H_j)$. This module is $0$ if $i>c$ or if $j>m+1-\grade(I)$.
The other one has as first screen
\[
 _1''E^i_j = \left\lbrace\begin{array}{ll} H^i_B (K_j) & \mbox{for all }i>r, \mbox{ and } j< q\\
						H^i_B (Z_q) & \mbox{for }q=j\\
						0 & \mbox{for all }i\leq r, \mbox{ and } j< q. 
\end{array}\right.
\]

From the second spectral sequence we deduce that, if $\mu\in \G$ is taken in such a way that $H^{i+k}_B (K_{q+k+1})_{\mu}$ vanishes for all $k\geq 0$, then $( _\infty''E^i_q)_{\mu}= H^{i}_B (Z_{q})_{\mu}$. Hence, if 
\begin{equation}\label{eqSuppHikK}
\mu\notin \bigcup_{k\geq 0}\Supp_\G(H^{i+k}_B(K_{k+q+1})) = \bigcup_{k\geq 0}(\Supp_\G(H^{k+i}_B(R)[-(k+q+1)\cdot\gamma]),
\end{equation}
then $( _\infty''E^i_q)_{\mu}= H^{i}_B (Z_{q})_{\mu}$

Comparing both spectral sequences, we have $\mu\notin \bigcup_{k\geq 0}\Supp_\G(H^{i+k}_B(H_{k+q}))$, we get $( _\infty''E^i_q)_{\mu}=0$.
This last condition is automatic for $i> c$ since $H^{i+k}_B(H_{k+q})=0$ for all $k\geq 0$. 
\end{proof}

\begin{cor}\label{LemMultiLCRGral2} Assume $f_0,\hdots,f_m\in R$ are homogeneous elements of degree $\gamma$. Write $I=(f_0,\hdots,f_m)$. Fix an integer $q$. If $\cd_B(R/I)\leq 1$, then the following hold
\begin{enumerate}
 \item for $i=0,1$, $\Supp_\G(H^i_B(Z_q))\subset (\SSup_B(\gamma)+(q-i)\cdot \gamma) \cup (\SSup_B(\gamma)+(q+1-i)\cdot \gamma)$.
 \item for $i>1$, $\Supp_\G(H^i_B(Z_q))\subset \SSup_B(\gamma)+(q+1-i)\cdot \gamma$.
\end{enumerate}
\end{cor}
\begin{proof} 
Since $\Supp_\G(H^{i+k}_B(H_{k+q}))\subset \SSup_B(\gamma)+(q-i)\cdot \gamma$ for all $k\geq 0$ by Remark \ref{RemEx1Koszul}, gathering together this with equation \eqref{eqSuppHikK} and Lemma \ref{LemMultiLCRGral}, the result follows.
\end{proof}

\begin{rem}\label{LemCyclesGral2} 
 We also have empty support for Koszul cycles in the following cases.
\begin{enumerate}
 \item if $\grade(B)\neq 0$, $H_B^{0} (Z_{p})=0$ for all $p>0$, and
 \item if $\grade(B)\geq 2$, $H^1_B(Z_p)=0$ for all $p>0$.
\end{enumerate}
\end{rem}
\begin{proof}
 The first claim follows from the inclusion $Z_p\subset K_p$ and the second from the exact sequence $0 \to Z_{p}\to K_{p}\to B_{p-1}\to 0$ that gives $0\to H_B^{0} (B_{p-1})\to H_B^{1} (Z_{p})\to H_B^{1} (K_{p})$, with $H_B^{0} (B_{p-1})$ as $B_{p-1}\subset K_{p-1}$.
\end{proof}

\section{$\G$-graded approximation complexes}\label{GPolyAppCom}

We treat in this part the case of a finitely generated abelian group $\G$ acting on a polynomial ring $R$. Write $R:=\kk[X_1,\hdots,X_n]$. Take $H$ a subgroup of $\ZZ^n$ and assume $\G= \ZZ^n/H$.

Take $m+1$ homogeneous elements $\f:=f_0,\hdots,f_m\in R$ of fixed degree $\gamma\in\G$. Set $I=(f_0,\hdots,f_m)$ the homogeneous ideal of $R$ defined by $\f$. Recall that $\RIR:=\bigoplus_{l\geq 0} (It)^l\subset R[t]$. It is important to observe that the grading in $\RIR$ is taken in such a way that the natural map $\alpha: R[T_0,\hdots,T_m] \to \RIR \subset R[t]: T_i\mapsto f_i t$ is of degree zero, and hence, $(It)^l\subset R_{l \gamma}\otimes_\kk \kk[t]_l$.

\medskip

Let $\T:=T_0,\hdots,T_m$ be $m+1$ indeterminates. There is a surjective map of rings $\alpha: R[\T] \twoheadrightarrow \RIR$ with kernel $\pp:=\ker(\alpha)$. 

\begin{rem}
 Observe that $\pp\subset R[\T]$ is $(\G\times\ZZ)$-graded. Set $\pp_{(\mu;b)}\subset R_{\mu}\otimes_\kk \kk[\T]_b$, and $\pp_{(*,0)}=0$. Denote $\bb:=(\pp_{(*,1)})=\paren{\{\sum g_iT_i :g_i\in R, \sum g_if_i=0\}}$. In other words, $\bb$ is $R[\T]$-ideal generated by $\textnormal{Syz}(\f)$.
\end{rem}

The natural inclusion $\bb\subset \pp$ gives a surjection $\beta:\SIR\cong R[\T]/\bb {\twoheadrightarrow} R[\T]/\pp \cong \RIR$ that makes the following diagram commute
\begin{equation}\label{diagReesSym}
 \xymatrix@1{ 
 0\ar[r] &\bb \ar[r] \ar@{^(->}[d] &R[\T] \ar[r] \ar@{=}[d] &\SIR \ar[r] \ar@{>>}[d]^{\beta} &0 \\
 0\ar[r] &\pp \ar[r] &R[\T] \ar[r]^{\alpha} &\RIR \ar[r] &0
}
\end{equation}

Set $\k.= \k. ^R(\f)$ for the Koszul complex of $\f$ over the ring $R$. Write $K_i:=\bigwedge^i R[-i \gamma]^{m+1}$, and $Z_i$ and $B_i$ for the $i$-th module of cycles and boundaries respectively. We write $H_i=H_i(\f;R)$ for the $i$-th Koszul homology module. 

We write $\Z.$, $\B.$ and $\M.$ for the approximation complexes of cycles, boundaries and homologies (cf. \cite{HSV1}, \cite{HSV2} and \cite{Va94}). Define $\Zc_\ell=Z_\ell[\ell\gamma]\otimes_R R[\T]$, where  $(Z_\ell[\ell\gamma])_{\mu}=(Z_\ell)_{\ell\gamma+\mu}$. Similarly we define $\Bc_\ell=B_\ell[\ell\gamma]\otimes_R R[\T]$ and $\Mc_\ell=H_\ell[\ell\gamma]\otimes_R R[\T]$. Note that since $R[\T]$ is $(\G\times\ZZ)$-graded, then also the complexes $\Z.$, $\B.$ and $\M.$ are $(\G\times\ZZ)$-graded.

Let us recall some basic facts about approximation complexes that will be useful in the sequel. In particular, recall that the ideal $J\subset R$ is said to be of \textit{linear type} if $\SIR\simeq \RIR$ (cf. \cite{Va94}).

\begin{defn}
 The sequence $a_1,\hdots,a_\ell$ in $R$ is said to be a proper sequence if $a_{i+1} H_j(a_1,\hdots,a_i;R)=0$, for all $0\leq i \leq \ell, 0 < j \leq i$. 
\end{defn}
Notice that an almost complete intersection ideal is generated by a proper sequence.

\medskip
Henceforward, we will denote $\HH_i:=H_i(\Z.)$ for all $i$.

\begin{lem}\label{LemaboutZ} With the notation above, the following statements hold:
\begin{enumerate}
 \item $\HH_0=\SIR$.
 \item $\HH_i$ is a $\SIR$-module for all $i$.
 \item If the ideal $I$ can be generated by a proper sequence then $\HH_i=0$ for $i>0$.
 \item If $I$ is generated by a $d$-sequence, then it can be generated by a proper sequence, and moreover, $I$ is of linear type.
\end{enumerate}
\end{lem}
\begin{proof}
 For a proof of these facts we refer the reader to \cite{Va94} or \cite{HSV2}.
\end{proof}

Assume the ideal $V(I)=V(\f)$ is of linear type off $V(B)$, that is, for every prime $\qq\not\supset B$, $(\SIR)_\qq=(\RIR)_\qq$. 
The key point of study is the torsion of both algebras as $\kk[\T]$-modules. Precisely, by the same arguments as in the homogeneous setting, we have the following result.

\begin{lem} With the notation above, we have
\begin{enumerate}
 \item $\ann_{\kk[\T]}((\RIR)_{(\nu,*)})=\pp\cap \kk[\T]=\ker(\phi^*)$, if $R_\nu\neq 0$;
 \item if $V(I)$ is of linear type off $V(B)$ in $\Spec(R)$, then
\[
 \SIR/ H^0_B(\SIR)=\RIR.
\]
\end{enumerate}
\end{lem}
\begin{proof}
 The first part follows from the fact that $\pp$ is $\G\times \ZZ$-homogeneous and as $\RIR$ is a domain, there are no zero-divisors in $R$. By localizing at each point of $\Spec(R)\setminus V(B)$ we have the equality of the second item.
\end{proof}

This result suggest that we can approximate one algebra by the other, when they coincide outside $V(B)$. 

\begin{lem}\label{ToricZacyclic1}
 Assume $B\subset \rad(I)$, then $\HH_i$ is $B$-torsion for all $i>0$.
\end{lem}
\begin{proof} 
 Let $\pp \in \Spec(R)\setminus V(B)$. In particular $\pp \in \Spec(R)\setminus V(I)$, hence, $(H_i)_\pp=0$. This implies that the complex $\M.$ (cf. \cite{HSV2}) is zero, hence acyclic, after localization at $\pp$. It follows that $(\Z.)_\pp$ is also acyclic \cite[Prop. 4.3]{BuJo03}.
\end{proof}

We now generalize Lemma \ref{ToricZacyclic1} for the case when $V(I)\nsubseteq V(B)$. The condition $B\subset \rad(I)$ can be carried to a cohomological one, by saying $\cd_B(R/I)=0$. Note that $V(I)$ is empty in $\Xc$, if $V(I)\subset V(B)$ in $\Spec(R)$, which is equivalent to $H^i_B(R/I)=0$ for $i>0$. This hypothesis can be relaxed by bounding $\cd_B(R/I)$. 

We will consider  $\cd_B(R/I)\leq 1$ for the sequel in order to have convergence of the horizontal spectral sequence $'E$ at step $2$.

Before getting into the next result, recall that $\Zc_q:=Z_q[q\cdot \gamma]\otimes_\kk \kk[\T]$. Furthermore, if $\grade(B)\geq 2$, it follows that 
\begin{equation}\label{EqSuppZcq}
 \Supp_{\G}(H_B^{k}(\Zc_{q+k}))=\Supp_{\G}(H_B^{k}(Z_{q+k}))-q\cdot \gamma\subset \SSup_B(\gamma)+(1-k)\cdot \gamma.
\end{equation}

Observe, that none of this sets on the right depend on $q$, thus, we define:

\begin{defn}\label{defRegion}
For $\gamma\in \G$, set 
\[
 \Region(\gamma):=\bigcup_{0<k< \min\{m,\cd_B(R)\}} (\SSup_B(\gamma)-k\cdot \gamma)\subset \G.
\]
\end{defn}

\begin{thm}\label{ToricZacyclic2}
Assume that $\grade(B)\geq 2$ and $\cd_B(R/I)\leq 1$. Then, if $\mu \notin \Region(\gamma)$,
\[
 H_B^i(\HH_{j})_{\mu}=0, \quad \mbox{for all }i,j.
\]

\end{thm}
\begin{proof} Consider the two spectral sequences that arise from the double complex $\check{\Cc}_B^\bullet\Z.$. Since $\supp (H_p)\subset I$, the first spectral sequence has at second screen $ _2'E^i_j = H^i_B (\HH_j) $. The condition $\cd_B(R/I)\leq 1$ gives that this spectral sequence stabilizes as $ _2'E^i_j = 0$ unless $j=0$ or, $i=0,1$ and $j>0$. 

The second spectral sequence has at first screen $'' _1 E^i_j= H^i_B(\Zc_j)$. Since $R[\T]$ is $R$-flat, $H^i_B(\Zc_j)=H^i_B(Z_j[j\gamma])\otimes_\kk \kk[\T]$. From and Remark \ref{LemCyclesGral2} the $H^i(\Zc_j)=0$ for $i=0,1$ and $j>0$. Comparing both spectral sequences, we deduce that the vanishing of $H_B^{k}(\Zc_{p+k})_{\mu}$ for all $k$, implies the vanishing of $H_B^{k}(\HH_{p+k})_{\mu}$ for all $k>1$.
Finally, from equation \eqref{EqSuppZcq} we have that if $\mu \notin \Region(\gamma)$ (which do not depend on $p$), then we obtain $H_B^i(\HH_j)_{\mu}=0$.
\end{proof}

\begin{lem}\label{CorToricZacyclic}
Assume $\grade(B)\geq 2$, $\cd_B(R/I)\leq 1$ and $I_\pp$ is almost a local complete intersection for every $\pp\notin V(B)$. Then, for all $\mu\notin \Region(\gamma)$, the complex $(\Z.)_{\mu}$ is acyclic and $H^0_B(\SIR)_{\mu}=0$.
\end{lem}
\begin{proof}
 Since $I_\pp$ is almost a local complete intersection for every $\pp\notin V(B)$, $\Z.$ is acyclic off $V(B)$. Hence, $\HH_q$ is $B$-torsion for all positive $q$. Since $\HH_q$ is $B$-torsion, $H^k_B(\HH_q)=0$ for $k>0$ and $H^0_B(\HH_q)=\HH_q$. From Theorem \ref{ToricZacyclic2} we have that $(\HH_q)_{\mu}=0$, and $H_B^0(\HH_0)_{\mu}=0$.
\end{proof}

\begin{rem}
 Observe that Theorem \ref{ToricZacyclic2} and Lemma \ref{CorToricZacyclic} provide acyclicity statements of some graded strands of the $\Zc$-complex, equivalently, acyclicity as a complex of sheaves over $\Tc$. In the multigraded case, acyclicity in ``good'' degree does not imply acyclicity of the complex, differently to the situation in the $\ZZ$-graded case (cf.\ \cite{BuJo03}).
\end{rem}


\medskip

\section{Implicitization of toric hypersurfaces in $\PP^n$}\label{ToricImplicitization}
In this section, we present and discuss the formal structure of the closed image of rational maps defined over a toric variety and how to compute an implicit equation of it. This subject has been investigated in several articles with many different approaches. The problem of computing the implicit equations defining the closed image of a rational map is an open research area with several applications.

\medskip

Let $\Xc$ be a non-degenerate toric variety over a field $\kk$, $\Delta$ be its fan in the lattice $N\cong \ZZ^d$ corresponding to $\Xc$, and write $\Delta(i)$ for the set of $i$-dimensional cones in $\Delta$ as before. Denote by $R$ the Cox ring of $\Xc$ (cf.\ Sec.\ \ref{CoxRing}).

Henceforward, we will focus on the study of the elimination theory \`a la Jouanolou-Bus\'e-Chardin (see for example \cite{BuJo03, BC05, BCJ06}). This aim brings us to review some basic definitions and properties.

Assume we have a rational map $\phi: \Xc \dto \PP^m$, defined by $m+1$ homogeneous elements $\f:=f_0,\hdots,f_m\in R$ of fixed degree $\gamma\in\D$. Precisely, any cone $\sigma\in \Delta$ defines an open affine set $U_\sigma$ (cf. \cite{Cox95}), and two elements $f_i, f_j$ define a rational function $f_i/f_j$ on some affine open set $U_\sigma$, and this $\sigma$ can be determine from the monomials appearing in $f_j$. In particular, if $\Xc$ is a multiprojective space, then $f_i$ stands for a multihomogeneous polynomial of multidegree $\gamma\in \ZZ_{\geq 0} \times \cdots \times \ZZ_{\geq 0}$.

We recall that for any $\D$-homogeneous ideal $J$, $\Proj_{\Xc}(R/J)$ simply stands for the gluing of the affine scheme $\Spec((R/J)_\sigma)$ on every affine chart $\Spec(R_\sigma)$, to $\Xc$. It can be similarly done to define from $\D\times \ZZ$-homogeneous ideals of $R\otimes_\kk \kk[\T]$, subschemes of $\Xc \times_\kk \PP^m$, and this projectivization functor will be denoted $\Proj_{\Xc \times \PP^m}(-)$. The graded-ungraded scheme construction will be denoted by $\Proj_{\Xc \times \AA^{m+1}}(-)$. For a deep examination on this subject, we refer the reader to \cite{Fu93}, and \cite{Cox95}. 

\begin{defn}\label{defToricSetting}
 Set $I:=(f_0,\hdots,f_n)$ ideal of $R$. Define $\Sc:=\Proj_{\Xc}(R/I)$ and $\Sc^{\textnormal{red}}:=\Proj_{\Xc}(R/\rad(I))$, the base locus of $\phi$. denote by $\Omega:=\Xc\setminus \Sc$, the domain of definition of $\phi$. 

Let $\Gamma_0$ denote the graph of $\phi$ over $\Omega$, and $\Gamma:=\overline{\Gamma_0}$ its closure in $\Xc \times \PP^m$. Scheme-theoretically we have $\Gamma=\Proj_{\Xc \times \PP^m}(\RIR)$, where $\RIR:=\bigoplus_{l\geq 0} (It)^l\subset R[t]$. 
\end{defn}

Recall that the two surjections, $R[\T]\twoheadrightarrow \SIR$ and $\beta:\SIR\twoheadrightarrow \RIR$, established on Diagram \ref{diagReesSym}, correspond to a chain of embedding $\Gamma \subset \Upsilon \subset \Xc \times \PP^m$, where $\Upsilon = \Proj_{\Xc \times \PP^m}(\SIR)$.

Assume the ideal $I$ is of linear type off $V(B)$, that is, for every prime $\qq\not\supset B$, $(\SIR)_\qq=(\RIR)_\qq$. Since the functors Sym and Rees commute with localization, $\Proj_{\Xc \times \PP^m}(\SIR)=\Proj_{\Xc \times \PP^m}(\RIR)$, that is $\Upsilon=\Gamma$ in $ \Xc \times \PP^m$. Moreover, $\Proj_{\Xc\times \AA^{m+1}}(\SIR)$ and $\Proj_{\Xc\times \AA^{m+1}}(\RIR)$ coincide in $ \Xc\times \AA^{m+1}$. Recall that this in general does not imply that $\SIR$ and $\RIR$ coincide, in fact this is almost never true: as $\RIR$ is the closure of the graph of $\phi$ which is irreducible, it is an integral domain, hence, torsion free; on the other hand, $\SIR$ is almost never torsion free.

\begin{rem}\label{gradeBinS}
 By definition of $B$, it can be assumed without loss of generality that $\grade (B)\geq 2$.
\end{rem}

Hereafter, we will assume that $\grade (B)\geq 2$.

\begin{lem}\label{lemDimyCDToric}
 If $\dim (V(I)) \leq 0$ in $\Xc$, then $\cd_B(R/I)\leq 1$.
\end{lem}
\begin{proof}
For any finitely generated $R$-module $M$ and all $i>0$, from Equation \eqref{equLC-SCforS} $H^i_{\ast}(\Xc,M^\sim)\cong H^{i+1}_B(M)$.
Applying this to $M=R/I$, for all $\gamma\in \D$ we get that
\[
 H^i(\Xc, (R/I)^\sim (\gamma))=H^i(V(I), \OO_{V(I)}(\gamma)),
\]
 that vanishes for $i>0$, since $\dim V(I) \leq 0$.
\end{proof}

\begin{lem}\label{lenDegZ}
Let $\Tc$ be a toric variety with Cox ring $R$, graded by the group $\G$, with irrelevant ideal $B$. Let $J\subset R$ be an homogeneous ideal of $R$. Write $Z=\Proj_\Xc(R/J)$, and assume that $\dim\ Z=0$ in $\Tc$, then, if $\mu\in \G$ is such that $H^0_B(R/J)_\mu=H^1_B(R/J)_\mu=0$ then $\dim (R/J)_\mu=\deg(Z)$.
\end{lem}
\begin{proof}
 Consider the exact sequence \eqref{equLC-SCforS} for $M=R/J$ in degree $\mu\in \G$ 
\[
 0\lto H^0_B(R/J)_\mu\lto (R/J)_\mu \lto H^0_{\ast}(\Tc,(R/J)^\sim(\mu))\lto H^{1}_B(R/J)_\mu\lto 0.
\]
 Since $Z$ is zero-dimensional, and $\Tc$ compact, assume $Z=Z_1\sqcup\cdots\sqcup Z_\ell$. We have that
$H^0_{\ast}(\Tc,(R/J)^\sim(\mu))\cong H^0_{\ast}(Z,\OO_Z(\mu))=\underset{1\leq i\leq \ell}{\bigoplus}H^0_{\ast}(Z_i,\OO_{Z_i}(\mu))$. Thus, since $H^0_{\ast}(Z_i,\OO_{Z_i}(\mu))=k^{\mult_\Tc(Z_i)}$, we conclude that 
\[
 \dim \paren{(R/J)_\mu}=\sum_{1\leq i\leq \ell}\mult_\Tc(Z_i)=\deg(Z).\qedhere
\]
\end{proof}

\begin{thm}\label{CorToricImplicit}
 Let $\Xc$ be a ($n-1$)-dimensional non-degenerate toric variety over a field $\kk$, and $R$ its Cox ring. Let $\phi: \Xc \dto \PP^{n}$ be a rational map, defined by $d+1$ homogeneous elements $f_0,\hdots,f_{d}\in R$ of fixed degree $\gamma\in\D$. Denote $I=\paren{f_0,\hdots,f_{n}}$. If $\dim(V(I))\leq 0$ in $\Xc$ and $V(I)$ is almost a local complete intersection off $V(B)$, then
\[
 \det((\Z.)_\mu)=H^{\deg(\phi)}\cdot G \in \kk[\T],
\]
for all $\mu\notin \Region(\gamma)$, where $H$ stands for the irreducible implicit equation of the image of $\phi$, and $G$ is relatively prime polynomial in $\kk[\T]$.
\end{thm}
\begin{proof}
 This result follows in the standard way, similar to the cases of implicitization problems in other contexts.

Recall that $\Gamma$ is the closure of the graph of $\phi$, hence, defined over $\Omega$. The bihomogeneous structure in $R\otimes_\kk \kk[\T]$ gives rise to two natural scheme morphisms $\pi_1$ and $\pi_2$: $\Xc \nto{\pi_1}{\ot} \Xc \times_{\kk} \PP^n \nto{\pi_2}{\to} \PP^n $. It follows directly that $\pi_2=\pi_1\circ\phi$ over the graph of $\phi$, $\pi_1^{-1}(\Omega)$.

From Theorem \ref{CorToricImplicit}, the complex of $\OO_{\PP^n}$-modules $(\Z.)^\sim$ is acyclic over $\Xc\times_\kk \PP^{n}$. We can easilly verify that this complex has support in $\Upsilon$, hence, $H^0(\Xc\times_\kk \PP^{n}, (\Z.)^\sim)=H^0(\Upsilon, (\Z.)^\sim)=\SIR$. Naturally, the factor $G$ defines a divisor in $\PP^{n}$ with support on $\pi_2(\Upsilon\setminus \Gamma)$, and $\Upsilon$ and $\Gamma$ coincide outside $\Sc\times \PP^{n}$. 

Following \cite{KMun}, due to the choice of $\mu\notin \Region(\gamma)$, one has:
\[
 [\det((\Z.)_\nu)]=\div_{\kk[\X]}(H_0(\Z.)_\mu)=\div_{\kk[\X]}(\SIR_\mu) =\sum_{\mbox{\scriptsize $\begin{array}{c}\qq \textnormal{ prime, }\\ \codim_{\kk[\X]}(\qq)=1 \end{array}$}}\length_{\kk[\X]_{\qq}} ((\SIR_\mu)_\qq)[\qq].
\]
Thus, for all $\mu\notin \Region(\gamma)$, we obtain
\[
[\det((\Z.)_\mu)]= \length_{\kk[\X]_{(H)}} ((\SIR_\mu)_{(H)})[(H)] +\hspace{-0.4cm}\underset{\mbox{\scriptsize $\begin{array}{c}\qq \textnormal{ prime,}\\ V(\qq) \not\subset V(H)\\ \codim_{\kk[\X]}(\qq)=1 \end{array}$}}{\sum} \hspace{-0.4cm}\length_{\kk[\X]_{\qq}} ((\SIR_\mu)_{\qq})[\qq]. 
\]
It follows that the first summand is the divisor associated to the irreducible implicit equation $H$, and the second one, the divisor associated to $G$. We write $\pp:=(H)$ and $\kkk(\pp):=R[\X]_{(\pp)}/(\pp)R[\X]_{\pp}$. We will show that $\length_{\kk[\X]_{\pp}} ((\SIR_\mu)_{\pp})= \deg(\phi)$. Observe that for every $\mu\in \G$ we have $(\SIR/H^0_B(\SIR)_\mu)_{\pp}=(\RIR_\mu)_{\pp}$. 
For every $\mu\in \G$ we have  
\[
 \length_{\kk[\X]_{\pp}} ((\RIR_\mu)_{\pp})=\dim_{\kkk(\pp)} \paren{\paren{\RIR\otimes_{R[\X]_{\pp}}\kkk(\pp)}_\mu}.
\]

Since $\deg(\pi_1)=1$, we have that $\dim_{\kkk(\pp)} \paren{\paren{\RIR\otimes_{R[\X]_{\pp}}\kkk(\pp)}_\mu}= \dim_{\kk} \paren{\paren{R/(\pi_1)_\ast(\pp R[\X]_\pp)}_\mu}$. The results follow from Lemma \ref{lenDegZ} taking $J=(\pi_1)_\ast(\pp R[\X]_\pp)$ and $\mu\notin \Region(\gamma)$.
 \end{proof}

We next give a detailed description of the extra factor $G$, as given in \cite[Prop. 5]{BCJ06}.

\begin{lem}\label{remExtraFactor}
 Let $\Xc$, $R$, $\phi: \Xc \dto \PP^{n}$, $H$ and $G$ be as in Theorem \ref{CorToricImplicit}. If $\kk$ is algebraically closed, then $G$ can be written as 
\[
 G=\prod_{\mbox{\scriptsize $\begin{array}{c}\qq \textnormal{ prime, } V(\qq) \not\subset V(H)\\ \codim_{\kk[\X]}(\qq)=1 \end{array}$}} L_\qq^{e_\qq-l_\qq}.
\]
in $\kk[\T]$, where $e_\qq$ stands for the Hilbert-Samuel multiplicity of $\SIR$ at $\qq$, and the number $l_{\qq}$ denotes $\length_{\kk[\X]_{\qq}}((\SIR_\mu)_{\qq})$.
\end{lem}
\begin{proof}
 The proof follows the same lines of that of \cite[Prop. 5]{BCJ06}. Since $R$ is a Cohen-Macaulay ring and for every $x\in \Sc$, the local ring $\OO_{\Tc,x}=R_x$ is Cohen-Macaulay, \cite[Lemma 6]{BCJ06} can be applied verbatim.
\end{proof}

The main idea behind Lemma \ref{remExtraFactor} is that only non-complete intersections points in $\Sc$ yield the existence of extra factors as in \cite{BCJ06}, \cite{BDD08} and \cite{Bot09}. Indeed, if $I$ is locally a complete intersection at $\qq\in \Sc$, then $I_\qq$ is of linear type, hence, $(\SIR)_\qq$ and $(\RIR)_\qq$ coincide. Thus, the schemes $\Proj_{\Xc \times \PP^m}(\SIR)$ and $\Proj_{\Xc \times \PP^m}(\RIR)$ coincide over $\qq$.


\medskip

\section{Multiprojective spaces and multigraded polynomial rings}\label{MultiGradedImplicitization}

In this section we focus on a better understanding of the multiprojective case. This is probably the most important family of varieties in the applications among toric varieties. Here we take advantage of the particular structure of the ring. This will permit us to make precise the acyclicity regions $\Region(\gamma)$ as subregions of $\ZZ^s$.

The problem of computing the implicit equation of a rational multiprojective hypersurface is one the most important among toric cases of implicitization. The theory follows as a particular case of the one developed in the section before, but many results can be better understood. In this case, the grading group is $\ZZ^s$, which permits a deeper insight in the search for a ``good zone'' $\Region$ for $\gamma$.

For the rest of this section, we will keep the following convention. Let $s$ and $m$ be fixed positive integers, $r_1\leq\cdots\leq r_s$ non-negative integers, and write $\x_i=(x_{i,0},\hdots,x_{i,r_i})$ for $1\leq i \leq s$. Write $R_i:= \kk[\x_i]$ for $1\leq i \leq s$, $R=\bigotimes_\kk R_i$, and $R_{(a_1,\hdots,a_s)}:=\bigotimes_\kk (R_i)_{a_i}$ stands for its bigraded part of multidegree $(a_1,\hdots,a_s)$. Hence, $\dim (R_i)=r_i+1,$ and $\dim (R)=r+s$, and $\prod_{1\leq i \leq s}\PP^{r_i}=\Mproj(R)$. Set $\aaa_i:=(\x_i)$ an ideal of $R_i$, and the irrelevant ideal $B:=\bigcap_{1\leq i \leq s} \aaa_i = \aaa_1\cdots\aaa_s$ defining the empty locus of $\Mproj(R)$. 
Let $f_0,\hdots,f_n\in \bigotimes_\kk \kk[\x_i]$ be multihomogeneous polynomials of multidegree $d_i$ on $\x_i$. Assume we are given a rational map
\begin{equation}\label{eqPhiMultiProj}
 \phi: \prod_{1\leq i \leq s}\PP^{r_i} \dto \PP^n: \x \mapsto (f_0:\cdots:f_n)(\x).
\end{equation}
Take $n$ and $r_i$ such that $n=1+\sum_{1\leq i \leq s}{r_i}$. 
Set $I:=(f_0,\hdots,f_m)$ for the multihomogeneous ideal of $R$, and $X=\Mproj(R/I)$ the base locus of $\phi$.

Set-theoretically, write $V(I)$ for the base locus of $\phi$, and $\Omega:=\prod_{1\leq i \leq s}\PP^{r_i}\setminus V(I)$ the domain of definition of $\phi$. Let $\Gamma_0$ denote the graph of $\phi$ over $\Omega$, and $\Gamma:=\overline{\Gamma_0}$ its closure in $(\prod_{1\leq i \leq s}\PP^{r_i}) \times \PP^m$. Scheme-theoretically we have $\Gamma=\Mproj(\RIR)$, where $\RIR:=\bigoplus_{l\geq 0} (It)^l\subset R[t]$. 

The grading in $\RIR$ is taken in such a way that the natural map 
\[
 \alpha: R[T_0,\hdots,T_m] \to \RIR \subset R[t]: T_i\mapsto f_i t
\]
is of degree zero, and hence $(It)^l\subset R_{(ld_1,\hdots,ld_s)}\otimes_\kk \kk[t]_l$.

\begin{rem}\label{lemDimyCDMultiProj}
 From Lemma \ref{lemDimyCDToric} we have that if $\dim (V(I)) \leq 0$ in $\PP^{r_1}\times\cdots\times\PP^{r_s}$, then $\cd_B(R/I)\leq 1$.
\end{rem}

We give here a more detailed description of the local cohomology modules $H_B^\ell(R)$ that is needed for the better understanding of the region $\Region(\gamma)$.

\subsection{The local cohomology modules $H_B^\ell(R)$}

Let $k$ be a commutative noetherian ring, $s$ and $m$ be fixed positive integers, $r_1\leq\cdots\leq r_s$ non-negative integers, and write $\x_i$, $R_i$ and $R$ as before. 

\begin{defn}
 We define $\check{R}_{i}:= \frac{1}{x_{i,0}\cdots x_{i,r_i}}k[x_{i,0}^{-1},\hdots,x_{i,r_i}^{-1}]$. Given integers $1\leq i_1<\cdots<i_t \leq s$, take $\alpha=\{i_1,\hdots,i_t\}\subset \{1,\hdots,s\}$, and set $\check{R}_\alpha:=\paren{\bigotimes_{j\in \alpha} \check{R}_{j}} \otimes_k \paren{\bigotimes_{j\notin \alpha} R_{j}}$.
\end{defn}

Observe that $\check{R}_{\{i\}}\cong \check{R}_{i} \otimes_k \bigotimes_{j\neq i} R_j$.

\begin{defn} Let $\alpha$ be a subset of $\{1,\hdots,s\}$. We define $Q_\emptyset:=\emptyset$. For $\alpha \neq \emptyset$, write $\alpha=\{i_1,\hdots,i_t\}$ with $1\leq i_1<\cdots<i_t \leq s$. For any integer $j$ write $\sg(j):=1$ if $j\in \alpha$ and $\sg(j):=0$ if $j\notin \alpha$. We define 
\[
 Q_\alpha:= \prod_{1\leq j \leq s} (-1)^{\sg(j)} \NN - \sg(j)r_j\ee_j \subset \ZZ^s,
\]
 the shift of the orthant whose coordinates $\{i_1,\hdots,i_t\}$ are negative and the rest are all positive. We set $\aaa_i$ for the $R$-ideal generated
 by the elements in $\x_i$, $B:=\aaa_1\cdots \aaa_s$, $\aaa_\alpha:=\aaa_{i_1}+\cdots+\aaa_{i_t}$ and $|\alpha|=r_{i_1}+\cdots+r_{i_t}$.
\end{defn}

\begin{lem}\label{lemSuppCheckRalpha}
For every $\alpha\subset \{1,\hdots,s\}$, we have $\Supp_{\ZZ^s}(\check{R}_{\alpha})=Q_\alpha$.
\end{lem}

\begin{rem}\label{remQalphaQbeta} 
For $\alpha,\beta \subset \{1,\hdots,s\}$, if $\alpha \neq \beta$, then
 $Q_\alpha \cap Q_\beta=\emptyset$.
\end{rem}

\begin{lem}\label{lemHalphaRalpha} 
Given integers $1\leq i_1<\cdots<i_t \leq s$, let $\alpha=\{i_1,\hdots,i_t\}$. There are graded isomorphisms of $R$-modules
\begin{equation}\label{eqCheckRalpha}
 H_{\aaa_\alpha}^{|\alpha|+\#\alpha} (R)\cong \check{R}_\alpha.
\end{equation}
\end{lem}
\begin{proof}
Recall that for any noetherian ring $R$ and any $R$-module $M$, if $x_0,\hdots,x_r$ are variables, then
\begin{equation}\label{etoile}
 H_{(x_0,\hdots,x_r)}^i(M[x_0,\hdots,x_r])= \left\lbrace\begin{array}{ll} 0 & \mbox{if }i\neq r+1\\
									\frac{1}{x_{0}\cdots x_{r}}M[x_{0}^{-1},\hdots,x_{r}^{-1}]  & \mbox{for } i=r+1. 
\end{array}\right.
\end{equation}
We induct on $|\alpha|$. The result is obvious for $|\alpha|=1$. Assume that $|\alpha|\geq 2$ and \eqref{eqCheckRalpha} holds for $|\alpha|-1$. Take $I=\aaa_{i_1}\cdots\aaa_{i_{t-1}}$ and $J=\aaa_{i_t}$. There is a spectral sequence $H^{p}_J (H^{q}_I (R)) \Rightarrow H^{p+q}_{I+J} (R)$. By \eqref{etoile}, $H^p_J (R)=0$ for $p\neq r_{i_t}+1$. Hence, the spectral sequence stabilizes in degree $2$, and gives $H^{r_{i_t}+1}_J (H^{|\alpha|-r_{i_t}-1}_I (R)) \cong H^{|\alpha|}_{I+J} (R)$. The result follows by applying \eqref{etoile} with $M=H^{|\alpha|-r_{i_t}-1}_I (R)$, and the inductive hypothesis.
\end{proof}

\begin{lem}\label{LemloccohR} With the above notation,

\begin{equation}\label{loccohR}
H_B^\ell (R) 
\cong \bigoplus_{\mbox{\scriptsize $\begin{array}{c}1\leq i_1<\cdots<i_t \leq s \\ r_{i_1}+\cdots+r_{i_t}+1=\ell\end{array}$}}H_{\aaa_{i_1}+\cdots+\aaa_{i_t}}^{r_{i_1}+\cdots+r_{i_t}+t} (R)
\cong \bigoplus_{\mbox{\scriptsize $\begin{array}{c}\alpha \subset \{1,\hdots,s\}\\ |\alpha|+1=\ell\end{array}$}}\check{R}_\alpha.
\end{equation}
\end{lem}
\begin{proof}
The second isomorphism follows from \ref{lemHalphaRalpha}. For proving the first isomorphism, we induct on $s$. The result is obvious for $s=1$. Assume that $s\geq 2$ and \eqref{loccohR} holds for $s-1$. Take $I=\aaa_1\cdots\aaa_{s-1}$ and $J=\aaa_s$. The Mayer-Vietoris long exact sequence of local cohomology for $I$ and $J$ is
\begin{equation}\label{eqSEL-MayerVie}
 \cdots \to H^\ell_{I+J}(R) \nto{\psi_\ell}{\to}  H^\ell_{I}(R)\oplus H^\ell_{J}(R) \to  H^\ell_{IJ}(R) \to  H^{\ell+1}_{I+J}(R) \nto{\psi_{\ell+1}}{\to}  H^{\ell+1}_{I}(R)\oplus H^{\ell+1}_{J}(R) \to \cdots.
\end{equation}
Remark that if $\ell\leq r_s$, then $H^\ell_{J}(R)=H^\ell_{I+J}(R)=H^{\ell+1}_{I+J}(R)=0$. Hence, $H^\ell_{B}(R)\cong H^\ell_{I}(R)$. Write $\tilde{R}:=R_1\otimes_k\cdots\otimes_k R_{s-1}$. Since the variables $\x_s$ does not appear on $I$, by flatness of $R_s$ and the last isomorphism, we have that $H^\ell_{B}(R)\cong H^\ell_{B}(\tilde{R})\otimes_k R_s$. In this case, the result follows by induction.

Thus, assume $\ell> r_s$. We next show that the map $\psi_\ell$ in the sequence \eqref{eqSEL-MayerVie} is the zero map for all $\ell$. Indeed, there is an spectral sequence $H^{p}_J (H^{q}_I (R)) \Rightarrow H^{p+q}_{I+J} (R)$. Since $H^p_J (R)=0$ for $p\neq r_s+1$, it stabilizes in degree $2$, and gives $H^{r_s+1}_J (H^{\ell-r_s-1}_I (R)) \cong H^{\ell}_{I+J} (R)$. We have graded isomorphisms 
\begin{equation}\label{eqIsosHJHI}
 H^{r_s+1}_J (H^{\ell-r_s-1}_I (R)) \cong H^{r_s+1}_J (H^{\ell-r_s-1}_I (\tilde{R})\otimes_k R_s)\cong (H^{\ell-r_s-1}_I (\tilde{R}))[\x_s^ {-1}]\cong H^{\ell-r_s-1}_I (\tilde{R})\otimes_k \check{R}_s,
\end{equation}
where the first isomorphism comes from flatness of $R_s$ over $k$, the second isomorphism follows from equation \eqref{etoile} taking $M=H^{\ell-r_s-1}_I (\tilde{R})$.
By \eqref{eqIsosHJHI} and the inductive hypothesis we have that.
\begin{equation}\label{eqHjHlR}
 H^{r_s}_J (H^{\ell-r_s-1}_I (R)) \cong 
\bigoplus_{\mbox{\scriptsize $\begin{array}{c}1\leq i_1<\cdots<i_{t-1} \leq s-1 \\ r_{i_1}+\cdots+r_{i_{t-1}}+1=\ell-r_s\end{array}$}}H_{\aaa_{i_1}+\cdots+\aaa_{i_{t-1}+t-1}}^{r_{i_1}+\cdots+r_{i_{t-1}}} (\tilde{R})\otimes_k \check{R}_s.
\end{equation}

Now, observe that the map $(H^{\ell-r_s-1}_I (\tilde{R}))[\x_s^ {-1}]\to H^\ell_{I}(R)\oplus H^\ell_{J}(R)$ is graded of degree $0$. Recall from Lemma \ref{lemSuppCheckRalpha} and \ref{lemHalphaRalpha}, and Remark \ref{remQalphaQbeta} we deduce $\Supp_{\ZZ^s}((H^{\ell-r_s-1}_I (\tilde{R}))[\x_s^ {-1}]) \cap \Supp_{\ZZ^s}(H^\ell_{I}(R)\oplus H^\ell_{J}(R))=\emptyset$. Thus, every homogeneous element on $(H^{\ell-r_s-1}_I (\tilde{R}))[\x_s^ {-1}]$ is necessary mapped to $0$.

Hence, for each $\ell$, we have a short exact sequence
\begin{equation}\label{secHIHJ}
 0 \to  H^\ell_{I}(R)\oplus H^\ell_{J}(R) \to  H^\ell_{IJ}(R) \to  H^{r_s+1}_{J} (H^{\ell+1-r_s-1}_{I} (R))\to 0.
\end{equation}
Observe that this sequence has maps of degree $0$, and for each degree $a\in \ZZ^s$ the homogeneous strand of degree $a$ splits. Moreover,
\[
 \Supp_{\ZZ^s}((H^{\ell-r_s}_I (\tilde{R}))[\x_s^ {-1}]) \sqcup \Supp_{\ZZ^s}(H^\ell_{I}(R)\oplus H^\ell_{J}(R))=\Supp_{\ZZ^s}(H^\ell_{IJ}(R)).
\]
Namely, every monomial in $H^\ell_{IJ}(R)$  comes from the module $(H^{\ell-r_s}_I (\tilde{R}))[\x_s^ {-1}]$ or it is mapped to $H^\ell_{I}(R)\oplus H^\ell_{J}(R)$ injectively, splitting the sequence \eqref{secHIHJ} of $R$-modules. Hence,
\[
 H^\ell_{B}(R)\cong H^\ell_{I}(R)\oplus H^\ell_{J}(R)\oplus H^{r_s+1}_{J} (H^{\ell-r_s}_{I} (R)).
\]
Now, $H^\ell_{I}(R) \cong H^\ell_{B}(\tilde{R})\otimes_k R_s$, $H^\ell_{J}(R)= 0$ if $\ell\neq r_s+1$ and $H^{r_s+1}_{J}(R)= \check{R}_s$. The result follows by induction and equation \eqref{eqHjHlR}.
\end{proof}

It follows from Section 4 that the ``good'' region in $\ZZ^s$ where the approximation complex $\Z.$ and the symmetric algebra $\SIR$ have no $B$-torsion is, for $\mu\in \ZZ^s$, 
\begin{equation}\label{eqRegion}
 \Region(\gamma):=\bigcup_{0<k< \min\{m,\cd_B(R)\}} (\SSup_B(\gamma)-k\cdot \gamma)\subset \ZZ^s,
\end{equation}
where $\SSup_B(\gamma):= \bigcup_{k\geq 0}(\Supp_{\ZZ^s}(H^{k}_B(R))+k\cdot\gamma)$. 

\begin{rem}\label{remInclusSupp}
It is clear that for $\gamma\in \NN^s$, then, $(\SSup_B(\gamma)-k\cdot \gamma)\supset (\SSup_B(\gamma)-(k+1)\cdot \gamma)$ for all $k\geq 0$. Thus, from Definition \ref{defRegion}, we see that 
for all $\gamma\in \NN^s$, 
\[
 \Region(\gamma)=\SSup_B(\gamma)-\gamma.
\]
\end{rem}

\begin{cor}\label{corRegionMgr} For $\gamma\in \ZZ^s$, we have that 
 \begin{equation}\label{eqRegionMgr}
 \Region(\gamma)=\bigcup_{\alpha \subset \{1,\hdots,s\}}(Q_\alpha +|\alpha|\gamma).
\end{equation}
\end{cor}

\begin{proof}
Combining Lemma \ref{LemloccohR} and Lemma \ref{lemSuppCheckRalpha} we have
\[
\Supp_{\ZZ^s}(H^{k}_B(R))= \bigcup_{\mbox{\scriptsize $\begin{array}{c}\alpha \subset \{1,\hdots,s\}\\ |\alpha|+1=k\end{array}$}}Q_\alpha.
\]
From Remark \ref{remInclusSupp} and the definition of $\SSup_B(\gamma)$ we have 
$
 \Region(\gamma)=\bigcup_{k\geq 1}(\Supp_{\ZZ^s}(H^{k}_B(R))+(k-1)\gamma).
$

Combining both results, we get that
\[
 \Region(\gamma)=\bigcup_{k\geq 1}\paren{ \bigcup_{\mbox{\scriptsize $\begin{array}{c}\alpha \subset \{1,\hdots,s\}\\ |\alpha|=k-1\end{array}$}}Q_\alpha +|\alpha|\gamma}=\bigcup_{\alpha \subset \{1,\hdots,s\}}(Q_\alpha +|\alpha|\gamma).\qedhere
\]
\end{proof}


\section{Examples}\label{Examples}
In this part we give an example where we determine the ``good zone'' of acyclicity of the $\Zc$-complex, and later we take an specific rational map. We start by summarizing our results for biprojective spaces.

Let $k$ be a field. Take $r\leq s$ two positive integers, and consider $\Xc$ to be the biprojective space $\PP^r_\kk\times \PP^s_\kk$. Take $R_1:=k[x_0,\hdots,x_r]$, $R_2:=k[y_0\hdots,y_s]$, and $\G:=\ZZ^2$. Write $R:=R_1\otimes_k R_2$ and set $\deg(x_i)=(1,0)$ and $\deg(y_i)=(0,1)$ for all $i$. Set $\aaa_1:=(x_0\hdots,x_r)$, $\aaa_2:=(y_0\hdots,y_s)$ and define $B:=\aaa_1\cdot \aaa_2 \subset R$ the irrelevant ideal of $R$, and $\mm:= \aaa_1+\aaa_2\subset R$, the ideal corresponding to the origin in $\Spec(R)=\AA^{r+s+2}_\kk$.

From Lemma \ref{LemloccohR} one has:
\begin{enumerate}
 \item $H^{r+1}_B(R) \cong H^{r+1}_{\aaa_1}(R)=\check{R}_{1}\otimes_k R_2$,
 \item $H^{s+1}_B(R) \cong H^{s+1}_{\aaa_2}(R)=R_1\otimes_k\check{R}_{2}$,
 \item $H^{r+s+1}_B(R) \cong H^{r+s+2}_{\mm}(R)=\check{R}_{\{1,2\}}$,
 \item $H^\ell_B(R)=0$ for all $\ell\neq r+1,s+1$ and $r+s+1$,
\end{enumerate}
if $r>s$ and $(1)$ and $(2)$ are replaced by $H^{r+1}_B(R) \cong H^{r+1}_{\aaa_1}(R)\oplus H^{s+1}_{\aaa_2}(R)=\check{R}_{1} \otimes_k R_2\oplus R_1\otimes_k\check{R}_{2}$ if $r=s$.

Let $\NN$ denote $\ZZ_{\geq 0}$. From Corollary \ref{corRegionMgr}, the subregion of $\ZZ^2$ of supports of the modules $H^\ell_B(R)$ can be described as follows that if $r>s$:
\begin{enumerate}
 \item $\Supp_{\ZZ^2}(H^{r+1}_B(R)) = Q_{\{1\}} = -\NN\times \NN-(r+1,0)$,
 \item $\Supp_{\ZZ^2}(H^{s+1}_B(R)) = Q_{\{2\}} = \NN\times -\NN-(0,s+1)$,
 \item $\Supp_{\ZZ^2}(H^{r+s+1}_B(R)) = Q_{\{1,2\}} = -\NN\times -\NN-(r+1,s+1)$,
\end{enumerate}
and $(1)$ and $(2)$ are replaced by $\Supp_{\ZZ^2}(H^{r+1}_B(R)) = Q_{\{1\}} \cup Q_{\{2\}}= (-\NN\times \NN-(r+1,0)) \cup  (\NN\times -\NN-(0,s+1))$ if $r=s$.

Assume we are given $r+s+2$ bihomogeneous polynomials $f_0,\hdots,f_{r+s+1}$ of bidegree $\gamma:=(a,b)\in \ZZ^2$, denote $n=r+s+1$ and define $I:=(f_0,\hdots,f_{n})$ ideal of $R$. Assume $\cd_B(R/I)\leq 1$, hence $\cd_B(H_i)\leq 1$ for all $i$. We have that 
\[
\begin{array}{rcl}
 \SSup_B(a,b)&=&(-\NN\times \NN-(r+1,0)+(r+1)(a,b))\cup\\
&&(\NN\times -\NN-(0,s+1)+(s+1)(a,b))\cup\\
&&(-\NN\times -\NN-(r+1,s+1)+(r+s+1)(a,b)),
\end{array}
\]
 from \eqref{eqRegion} and \ref{corRegionMgr} we have that $\Region(a,b)=\SSup_B(a,b)-(a,b)$, as is shown in the picture below, we obtain the formula
\[
\begin{array}{rcl}
 \Region(a,b)&=& \paren{Q_{\{1\}}+r(a,b)}\cup \paren{Q_{\{2\}}+s(a,b)}\cup \paren{Q_{\{1,2\}}+(r+s)(a,b)}\\
&=&(-\NN\times \NN+(ra-r-1,rb))\cup\\
&&(\NN\times -\NN+(sa,sb-s-1))\cup\\
&&(-\NN\times -\NN+(ra+sa-r-1,rb+sb-s-1)),
\end{array}
\]

Since $s\leq r$, we have that the region $\Aa$ of vanishing is:
\[
 \Aa:=\complement\Region(a,b) = \NN^2+(ra-r,rb+sb-s)\cup \NN^2+(ra+sa-r,sb-s).
\]

\begin{center}
 \includegraphics[scale=1]{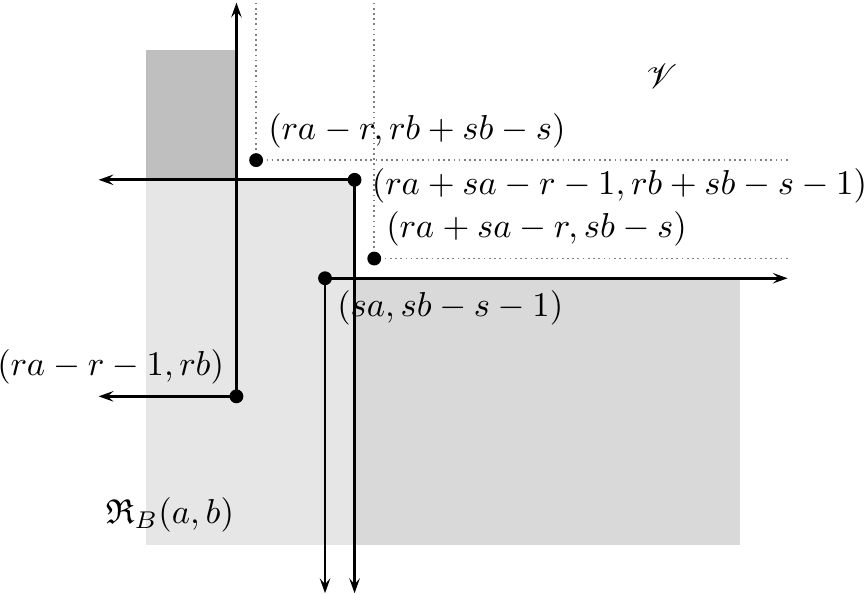}
\end{center}

Taking $\mu\in \Aa$, the $\Zc$-complex associated to $f_0,\hdots,f_n$ in degree $\nu$ is acyclic and $\Sym(f_0,\hdots,f_5)$ has no $B$-torsion. We conclude that we can compute the implicit equation of $\phi$ as a factor of $\det((\Z.)_{(\mu,\ast)})$ for $\mu\in \Aa$.

\subsubsection{Bigraded surfaces}

As it is probably the most interesting case from a practical point of view, and has been widely studied (see e.g.\ \cite{Co03}), we restrict our computations to parametrizations of a bigraded surface given as the image of a rational map $\phi: \PP^1\times\PP^1 \dashrightarrow \PP^3$ given by $4$ homogeneous polynomials of bidegree $(a,b)\in\ZZ^2$. Thus, in this case, the $\Zc$-complex can be easily computed, and the region $\Region(a,b)$ where are supported the local cohomology modules
\begin{equation*}\label{Regione1e2}
 \Region(a,b)=(\Supp_{\ZZ^2}(H^{2}_B(R))+(a,b))\cup(\Supp_{\ZZ^2}(H^{3}_B(R))+2\cdot(a,b)).
\end{equation*}
This case 

Let $k$ be a field. Assume $\Tc$ is the biprojective space $\PP^1_\kk\times \PP^3_\kk$. Take $R_1:=k[x_1,x_2]$, $R_2:=k[y_1,y_2,y_3,y_4]$. Write $R:=R_1\otimes_k R_2$ and set $\deg(x_i)=(1,0)$ and $\deg(y_i)=(0,1)$ for all $i$. Set $\aaa_1:=(x_1,x_2)$, $\aaa_2:=(y_1,y_2,y_3,y_4)$ and define $B:=\aaa_1\cdot \aaa_2 \subset R$ the irrelevant ideal of $R$, and $\mm:= \aaa_1+\aaa_2\subset R$, the ideal corresponding to the origin in $\Spec(R)$.

By means of the Mayer-Vietoris long exact sequence in local cohomology, we have that:

\begin{enumerate}
 \item $H^2_B(R) \cong (\check{R}_{1}\otimes_k R_2) \otimes (R_1 \otimes_k \check{R}_{2})$,
 \item $H^3_B(R) \cong H^4_{\mm}(R)=\check{R}_{\{1,2\}}$,
 \item $H^\ell_B(R)=0$ for all $\ell\neq 2$ and $3$. 
\end{enumerate}

Thus, we get that:

\begin{enumerate}
 \item $\Supp_{\ZZ^2}(H^2_B(R)) = -\NN\times \NN+(-2,0)\cup\NN\times -\NN+(0,-2)$.
 \item $\Supp_{\ZZ^2}(H^3_B(R)) = -\NN\times -\NN+(-2,-2)$, .
\end{enumerate}

\begin{equation}
 \Region(a,b)=(-\NN\times \NN+(a-2,b))\cup(\NN\times -\NN+(a,b-2))\cup(-\NN\times -\NN+(2a-2,2b-2)).
\end{equation}

\begin{center}
 \includegraphics[scale=1]{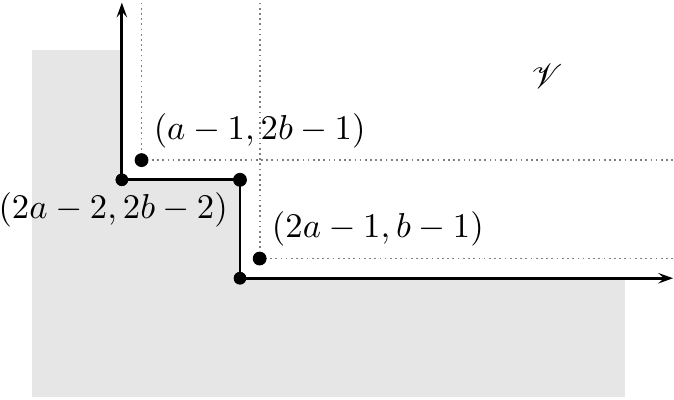}
\end{center}

If $\nu\notin  \Region(a,b)$, then the complex
\begin{equation*}
\ 0 \to (\Zc_3)_{(\nu,*)}(-3) \to (\Zc_{2})_{(\nu,*)}(-2) \to (\Zc_1)_{(\nu,*)}(-1) \nto{M_\nu}{\lto} (\Zc_0)_{(\nu,*)}\to 0 .
\end{equation*}
is acyclic and $H_B^0(H_0(\Zc_\bullet)_\nu)=H_B^0(\SIR)_\nu=0$.

\subsection{Implicitization with and without embedding}

The following example illustrates how the method of computing the implicit equation by means of approximation complexes works in one example.

\begin{exmp} Assume that we take the Newton polytope $\Nc(f)= \conv(\{(0,0),(a,0),(0,b),(a,b)\})$ and consider $\Tc= \Tc_{\Nc(f)}$ the toric variety associated to $\Nc(f)$ (for a wider reference see \cite{BD07,BDD08,Bot09}). We easily see that $\Tc\cong \PP^1\times \PP^1$ embedded in $\PP^{(a+1)(b+1)-1}$ via a Segre-Veronese embedding.

We define the ring with bihomogeneous coordinates $(s:u)$ and $(t:v)$ and the polynomials $f_0,f_1,f_2,f_3$ defining the rational map. And we copy the code in the algebra software Macaulay2 \cite{M2}

{\footnotesize \begin{verbatim}
i1 : QQ[s,t]
i2 : f0 = 3*s^2*t-2*s*t^2-s^2+s*t-3*s-t+4-t^2;
i3 : f1 = 3*s^2*t-s^2-3*s*t-s+t+t^2+t^2+s^2*t^2;
i4 : f2 = 2*s^2*t^2-3*s^2*t-s^2+s*t+3*s-3*t+2-t^2;
i5 : f3 = 2*s^2*t^2-3*s^2*t-2*s*t^2+s^2+5*s*t-3*s-3*t+4-t^2;
\end{verbatim}}

The good bound $\nu_0$ that follows from \cite{BD07,BDD08,Bot09} is $\nu_0=2$. We compute with the algorithm developed in \cite{BD10} the matrix $M_\nu$ for $\nu=2$:
{\footnotesize \begin{verbatim}
i7 : nu=2;
i8 : L={f0,f1,f2,f3};
i9 : rM=representationMatrix (teToricRationalMap L,nu);
                                25                          51
o9 : Matrix (QQ[X , X , X , X ])   <--- (QQ[X , X , X , X ])
                 0   1   2   3               0   1   2   3
\end{verbatim}}
It can be easily checked that it is a matrix that represents the surface (with perhaps some other extra factor), since it ranks drops over the surface:
{\footnotesize \begin{verbatim}
i10 : R=QQ[s,t,X_0..X_3];
i11 : Eq=sub(sub(rM,R), {X_0=>(sub(f0,R)), X_1=>(sub(f1,R)), X_2=>(sub(f2,R)), X_3=>(sub(f3,R))});
              25       51
o11 : Matrix R   <--- R
i12 : rank rM
o12 = 25
i13 : rank Eq
o13 = 24
\end{verbatim}}
\end{exmp}

\begin{exmp}\label{exmpBigraded} Without a toric embedding.
 Let us consider the bigraded ring $S$ as follows (we put three coordinates since we also keep the total degree):

{\footnotesize \begin{verbatim}
i1 : S=QQ[s,u,t,v,Degrees=>{{1,1,0},{1,1,0},{1,0,1},{1,0,1}}];
i2 : f0 = 3*s^2*t*v-2*s*u*t^2-s^2*v^2+s*u*t*v-3*s*u*v^2-u^2*t*v+4*u^2*v^2-u^2*t^2;
i3 : f1 = 3*s^2*t*v-s^2*v^2-3*s*u*t*v-s*u*v^2+u^2*t*v+u^2*t^2+u^2*t^2+s^2*t^2;
i4 : f2 = 2*s^2*t^2-3*s^2*t*v-s^2*v^2+s*u*t*v+3*s*u*v^2-3*u^2*t*v+2*u^2*v^2-u^2*t^2;
i5 : f3 = 2*s^2*t^2-3*s^2*t*v-2*s*u*t^2+s^2*v^2+5*s*u*t*v-3*s*u*v^2-3*u^2*t*v+4*u^2*v^2-u^2*t^2;
\end{verbatim}}

As we have seen in equation \eqref{Regione1e2}
\begin{equation*}
 \Region(a,b)=(-\NN\times \NN+(a-2,b))\cup(\NN\times -\NN+(a,b-2))\cup(-\NN\times -\NN+(2a-2,2b-2)).
\end{equation*}
Hence, since $(a,b)=(2,2)$,
\[
 \Aa:=\complement\Region(a,b)=\NN^2+(a-1,2b-1)\cup \NN^2+(2a-1,b-1)=\NN^2+(1,3)\cup \NN^2+(3,1).
\]

Thus, the good bounds here are $\nu_0=(1,3)$ or $\nu_0=(3,1)$. We compute with the algorithm developed in \cite{Bot10M2} the matrix $M_\nu$ for $\nu=(3,1)$:
{\footnotesize \begin{verbatim}
i6 : L = {f0,f1,f2,f3};
i7 : nuu1 = {4,3,1};
i8 : rM = representationMatrix (L,nuu1);
                                8                           8
o8 : Matrix (QQ[X , X , X , X ])   <--- (QQ[X , X , X , X ])
                 0   1   2   3               0   1   2   3
\end{verbatim}}
And we verify that it is a matrix that represents the surface (with perhaps some other extra factor) by substituting on the parametrization of the surface:
{\footnotesize \begin{verbatim}
i9 : R = QQ[s,u,t,v,X_0..X_3];
i10: Eq = sub(sub(rM,R), {X_0=>(sub(f0,R)), X_1=>(sub(f1,R)), X_2=>(sub(f2,R)), X_3=>(sub(f3,R))});
              8        8
o10 : Matrix R   <--- R
i11 : rank rM
o11 = 8
i12 : rank Eq
o12 = 7
\end{verbatim}}
We can compute the implicit equation by taking gcd of the maximal minors of $M_\nu$ or just taking one minor and factorizing. We can compute one determinant and obtaining the implicit equation by doing
{\footnotesize \begin{verbatim}
i13 : eq = implicitEq ({f0,f1,f2,f3},{5,3,2})}
               8             7               6 2            5 3           4 4            3 5          2 6       
o13 = 63569053X  - 159051916X X  + 175350068X X  - 82733240X X  + 2363584X X  + 14285376X X  + 139968X X  ...
               0             0 1             0 1            0 1           0 1            0 1          0 1    
\end{verbatim}}

\end{exmp}

\subsubsection{On the size of the matrices}

Consider the polytope $\Nc(f)= \conv(\{(0,0),(a,0),(0,b),(a,b)\})$ as is the case in the example above, and consider $\Tc= \Tc_{\Nc(f)}$ the toric variety associated to $\Nc(f)$ (for a wider reference see \cite{BD07,BDD08,Bot09}). We will compare the matrices obtained by means of a toric embedding in \cite{BD07,BDD08,Bot09}  with respect to the matrix $M_\nu$ we get without an embedding, in the biprojective case. Thus $\Tc$ is a bidimensional surface embedded in $\PP^{(a+1)(b+1)-1}$. We obtain a map $g: \Tc \dto \PP^3$ given by homogeneous polynomials of degree $1$. Hence, for $\nu_0=2$ we have that
\[
 M_2 \in \mat_{h_{R}(2),h_{Z_1}(3)}(\kk[X_0,X_1,X_2,X_3]),
\]
and $h_{R}(2)=(2a+1)(2b+1)$. This is the case above: $a=2$, $b=2$ then $h_{R}(2)=(2a+1)(2b+1)=5\cdot 5=25$.

\medskip

In the bigraded setting (assume $a\leq b$), 
\[
 M_2 \in \mat_{h_{R}(2a-1,b-1),h_{Z_1}(3a-1,2b-1)}(\kk[X_0,X_1,X_2,X_3])
\]
and $h_{R}(2a-1,b-1)= \dim(\kk[s,u]_{2a-1}) \dim(\kk[t,v]_{b-1})=2ab$. This is the case above: $a=2$, $b=2$ then $h_{R}(2)=2\cdot2\cdot2=8$.

\medskip
Next result gives a more detailed analysis of the size of the matrix $M_\nu$.

\begin{lem}\label{lemSizeMnu}
 Given a finite rational map $\phi: \PP^1\times\PP^1 \dashrightarrow \PP^3$ with finitely many laci base points (or none), given by $4$ homogeneous polynomials $f_0,f_1,f_2,f_3$ where $f_i\in R_{(a,b)}$. Take $\nu=(2a-1,b-1)$ (equiv.\ with $\nu=(a-1,2b-1)$). Write $\Delta_\nu=\det((\Z.)_\nu)$ for the determinant of the $\nu$-strand, then
 \[
  \deg(\Delta_\nu)=2ab-\dim\paren{(H_2)_{(4a-1,3b-1)}}
 \]
 and the matrix $M_\nu$ is square of size $2ab\times 2ab$ iff $(H_2)_{(4a-1,3b-1)}=0$.
\end{lem}
\begin{proof}
 Let $\Z.: 0 \to \Zc_3\to \Zc_2\to \Zc_1\to \Zc_0\to 0$ be the $\Zc$-complex, with $\Zc_i=Z_i(ia,ib)\otimes_k k[\X]$. Take $\nu=(2a-1,b-1)$. As $b-1<\deg_{t,v}(f_i)$ for all $i$, $(B_i(ia,ib))_\nu=0$ for all $i$, hence $(Z_i)_{i(a,b)+\nu}=(H_i)_{i(a,b)+\nu}$ for all $i$. This is, $(\Z.)_\nu= (\M.)_\nu$. Since $\depth(I)\geq 2$, $H_i=0$ for $i>2$, thus $(Z_3)_{3(a,b)+\nu}=0$. Since the complex
\[
 (\Z.)_{(2a-1,b-1)}:0 \to (H_2)_{(4a-1,3b-1)} [\X] \nto{N_\nu}{\lto} (H_1)_{(3a-1,2b-1)} [\X] \nto{M_\nu}{\lto} R_{(2a-1,b-1)}[\X]\to 0
\]
 is acyclic and the entries of $M_\nu$ and $N_\nu$ are linear on $X_i$'s, $\deg(\Delta_\nu)=\dim (R_\nu)-\dim\paren{(H_2)_{(4a-1,3b-1)}}$ and $M_\nu$ is square of size $2ab\times 2ab$ iff $(H_2)_{(4a-1,3b-1)}=0$.
\end{proof}

We can compute the degree of the image as follows.

\begin{lem}\label{lemDegImage}
 Given a finite rational map $\phi: \PP^1\times\PP^1 \dashrightarrow \PP^3$, given by $4$ homogeneous polynomials $f_0,f_1,f_2,f_3$ defining an ideal $I$, where $f_i\in R_{(a,b)}$. Assume that $\Pc:=\proj(R/I)$ is finite and laci. Then,
\[
 \deg(\phi)\deg(\Hc)=2ab-\sum_{x\in \Pc}e_x,
\]
where $e_x=e(I^\sim_x,\OO_{\PP^1\times \PP^1,x})$ is the multiplicity at $x$.
\end{lem}
\begin{proof}
Set $\LL:=\OO_{\PP^1\times \PP^1}(a,b)$, $\Gamma$ for the blow-up of $\PP^1\times \PP^1$ along $I$ and $\pi: \Gamma\to \PP^1\times \PP^1$ the natural projection. Since $\deg(\pi_\ast[\Gamma])=\deg(\phi)\deg(\Hc)$, by \cite[Prop.\ 4.4]{FuIT}, one has
 \[
 \deg(\phi)\deg(\Hc)=\int_{\PP^1\times \PP^1}c_1(\LL)^2-\int_\Pc (1+c_1(\LL))^2\cap s(\Pc,\PP^1\times\PP^1),
\]
where $c_1(\LL)$ is the first Chern class of $\LL$, and $s(\Pc,\PP^1\times\PP^1)$ the Segre class of $\Pc$ on $\PP^1\times \PP^1$. Since $\Pc$ has dimension zero, $C_1(\LL)=0$ on $\Pc$. It is well known that $\int_{\PP^1\times \PP^1}c_1(\LL)^2=2ab$. From \cite[Sec.\ 4.3]{FuIT}, $s(\Pc,\PP^1\times\PP^1)=\sum_{x\in \Pc}e_x$.
\end{proof}

\begin{rem}
 Combining Theorem \ref{CorToricImplicit} and Lemma \ref{remExtraFactor} we have that if $\Pc$ is lci, then $\Delta_\mu=H^{\deg(\phi)}$ for all $\mu\notin \Region(\gamma)$. From Lemma \ref{lemSizeMnu} and \ref{lemDegImage} we have that $\deg(\Delta_\nu)=\deg(H^{\deg(\phi)})=\deg(\phi)\deg(\Hc)=2ab-\sum_{x\in \Pc}e_x$. In particular $\dim\paren{(H_2)_{(4a-1,3b-1)}}=\sum_{x\in \Pc}e_x$.
\end{rem}


\medskip

\subsection*{Acknowledgments\markboth{Acknowledgments}{Acknowledgments}}
 I would like to thank my thesis advisors, Marc Chardin and Alicia Dickenstein, for the great help I have received from them. Also, I want to thank Professor J.-P. Jouanolou and Professor D.A. Cox for the very detailed reading of my thesis, for their reports and in particular their comments on the chapter related to this paper. Finally, to Laurent Bus\'e for his comments and questions before, during and after the defense as a member of the jury, and for motivating my interest in this area.

\medskip


\def\cprime{$'$}

\end{document}